\documentclass[11pt,reqno]{amsart}
\usepackage{amsmath, amssymb, amsthm}
\usepackage{xcolor}

   \makeatletter
\def\proof{\@ifnextchar[{\@oproof}{\@nproof}}
\def\@oproof[#1][#2]{\trivlist\item[\hskip\labelsep\textit{#2 Proof of\
#1.}~]\ignorespaces}
\def\@nproof{\trivlist\item[\hskip\labelsep\textit{Proof.}~]\ignorespaces}

\numberwithin{equation}{section}
\newtheorem{theorem}{Theorem}

\begin{document}
\author{Ajai Choudhry and Bibekananda Maji}

\address{Ajai Choudhry\\
13/4,  Clay Square,  Lucknow 226001, India.}
\email{ajaic203@yahoo.com}

\address{Bibekananda Maji\\ Department of Mathematics \\
Indian Institute of Technology Indore \\
Simrol,  Indore,  Madhya Pradesh 453552, India.} 
\email{bibek10iitb@gmail.com,  bmaji@iiti.ac.in}

\thanks{$2020$ \textit{Mathematics Subject Classification.} Primary 11D09; Secondary 11E25, 11N36,  11N37.\\
\textit{Keywords and phrases.} Sums of two squares; consecutive integers; arithmetic progessions.}

\Large \title{Finite sequences of  integers expressible as sums of two squares}

\date{}

\begin{abstract}
This paper is concerned with finite sequences of integers that may be written  as sums of  squares of two nonzero integers. We first find infinitely many integers $n$ such that  $n, n+h$ and $n+k$ are all  sums of two squares where $h$ and $k$ are  two arbitrary integers,  and as an immediate corollary  obtain, in parametric terms, three consecutive integers that  are  sums of two squares. Similarly we obtain $n$ in parametric terms such that all the four integers  $n,  n+1,  n+2,  n+4$ are sums of two squares. We also find infinitely many integers $n$ such that all the five integers $n,  n+1,  n+2,  n+4, n+5$ are sums of two squares, and finally, we find infinitely many arithmetic progressions, with common difference $4$,  of five integers  all of which  are sums of two squares. 
\end{abstract}



\maketitle
\section{Introduction}
This paper is concerned with finite sequences of integers that are all expressible as sums of two perfect squares of integers. It may be recalled in this context that the problem of finding three consecutive integers, all expressible as sums of two squares, dates back at least to 1903 when the first numerical solutions to the problem were found \cite[p.\ 252]{Di}. Subsequently  a few   one-parameter solutions have been found \cite{CD}. The online encyclopedia of integer sequences also contains a reference to the sequence of integers $k$ such that $k, k+1$ and $ k+2$ are all sums of two squares \cite{OE}.

A related problem posed by Littlewood asks whether for given distinct positive integers $h$ and $k$, there exist infinitely many integers $n$ such that $n, n+h, n+k$ are all sums of two squares. Hooley \cite{Ho} provided an affirmative answer to this question without, however, giving an explicit construction of such integers.

We give an alternative solution to Littlewood's problem. In fact, given any two integers $h$ and $k$, whether positive or negative, we find a three-parameter solution to the problem of finding $n$ such that $n, n+h$, and $n+k$ are all expressible as sums of two  squares of nonzero integers. This solution immediately yields a three-parameter solution to the problem of finding integers $n$ such that the three consecutive integers $n-1, n$ and $n+1$ are sums of two nonzero squares.

We also explore longer sequences of integers that are all expressible as sums of two squares. It has been shown in \cite{CD} that any triple of consecutive integers that are sums of two squares must start with a multiple of 8. Further, since a perfect square can only be congruent to 0 or 1 modulo 4, any integer congruent to 3 modulo 4 cannot be a sum of two squares.  It follows that if $n, n+1$ and $n+2$ are all expressible as sums of two squares, $n+3$ cannot be so expressed,  and if, in addition,   $n+4$ and $n+5$ are  also  sums of two squares, then $n+6$ and $n+7$ cannot be  sums of two squares.

We obtain in this paper infinitely many integers $n$ such that the five integers $n, n+1, n+2, n+4, n+5$ can all be written as  sums of squares of two nonzero integers.  Similarly,  we obtain infinitely many integers $n$ such that the  five integers in the arithmetic progression $n +4j, j=0, \ldots, 4$, are all sums of squares of two nonzero integers.

\section{Infinitely many integers $n$ such that $n, n+h$ and $n+k$ are all sums of two squares}\label{Mainth}
We will now give a constructive solution to the problem posed by Littlewod as mentioned in the Introduction.
\begin{theorem}\label{Thmhk} Given two arbitrary integers $h$ and $k$, whether positive or negative, there exist infinitely many integers $n$ such that the integers $n, n+h$ and $n+k$ are all expressible as sums of  squares of two nonzero integers. 
\end{theorem}
\begin{proof} We will actually construct infinitely many integers $n$  such that $n, n+h$ and $n+k$ are all sums of two squares. To obtain such integers,  we will solve the following three simultaneous equations 
in integers:
\begin{align}
n & = x^2+y^2, \label{eqn} \\
n+h & = (x+u_1)^2+(y+u_2)^2, \label{eqnh} \\
n+k &= (x+v_1)^2+(y+v_2)^2. \label{eqnk}
\end{align}

Eliminating $n$ first between Eqs. \eqref{eqn} and \eqref{eqnh}, and then between Eqs. \eqref{eqn} and \eqref{eqnk}, we get the following two equations, respectively:
\begin{align}
u_1^2 + 2u_1x + u_2^2 + 2u_2y - h& =0, \label{eqh} \\
v_1^2 + 2v_1x + v_2^2 + 2v_2y - k& =0. \label{eqk}
\end{align}

Now Eqs. \eqref{eqh} and \eqref{eqk} may be considered as two linear equations in $x$ and $y$, and on solving them, we get,
\begin{equation}
\begin{aligned}
x & = -(u_1^2v_2 + u_2^2v_2 - u_2v_1^2 - u_2v_2^2 - hv_2 + ku_2)\\
 & \quad \quad \times \{2(u_1v_2 - u_2v_1)\}^{-1}, \\
y & = (u_1^2v_1 - u_1v_1^2 - u_1v_2^2 + u_2^2v_1 - hv_1 + ku_1)\\
   & \quad \quad \times \{2(u_1v_2 - u_2v_1)\}^{-1}. 
\end{aligned}
\label{valxy}
\end{equation}

We have to  choose integer values for $u_i, v_i, i=1, 2$, such that the values of $x$ and $y$ given by \eqref{valxy} are integers. It seems reasonable to try integer values of $u_i, v_i, i=1, 2$, such that  $u_1v_2 - u_2v_1$ is either 1 or 2. We may thus  choose
\begin{equation}
u_1 = ft + 1, u_2 = t, v_1 = fgt + f + g, v_2 = gt + 1, \label{valuvtry1}
\end{equation}
when $u_1v_2 - u_2v_1=1$ or we may choose
\begin{equation}
u_1 = ft + 1, u_2 = t, v_1 = fgt + 2f + g, v_2 = gt + 2, \label{valuvtry2}
\end{equation}
when $u_1v_2 - u_2v_1=2$, 
where in both cases $f, g$ and $t$ are arbitrary integer parameters. 

We will still have $2$ or $2^2$ in the denominator of the values of $x$ and $y$ but we will show that we can always choose the parameters $f, g$ and $t$ such that we get integer values for $x$ and $y$. As we shall see later, it is not necessary to choose the parameters $u_i, v_i, i=1, 2$, such that  $u_1v_2 - u_2v_1$ is either 1 or 2 but the values of $u_i, v_i$, given by \eqref{valuvtry1} or \eqref{valuvtry2} with suitable choices of the parameters $f, g, t$ suffice to prove our theorem.

We have to  consider, without loss of generality, the following three cases:

(i) If  both integers $h$ and $k$ are odd, we use the values of  $u_i, v_i, i=1, 2$, as defined by \eqref{valuvtry1} where we  take all the three parameters $f, g$ and $t$ to be arbitrary even integers when the values of both $x$ and $y$ given by \eqref{valxy} are readily seen to be integers.

(ii)  If  one of the two integers $h$ and $k$ is odd and the other is even, we may assume that $h$ is even and $k$ is odd. We again take $u_i, v_i, i=1, 2$, as defined by \eqref{valuvtry1}  and  we now take both $f, g$ as arbitrary even integers and $t$ as an arbitrary odd integer,  and   again the values of both $x$ and $y$ given by \eqref{valxy} are readily seen to be integers. 

(iii) We now consider the case when both $h$ and $k$ are even integers. First we note that if $h$ and $k$ have a common squared factor, that is, there exist integers $m, h_1$ and $k_1$, such that $h=m^2h_1, k=m^2k_1$, where $h_1, k_1$ do not have a common squared factor, and there exist infinitely many integers $n$ such that $n, n+h_1, n+k_1$ are all sums of two squares, then we immediately get infinitely many integers $m^2n$ such that $m^2n, m^2n+h$ and $m^2n+k$ are all sums of two squares.  There is thus no loss of generality in  assuming that the integers $h$ and $k$ do not have a common squared factor and  accordingly, we assume that $h$ and $k$ do not have 4 as a common factor. 

When   $h$ and $k$ do not have a common factor $4$, we may consider, without loss of generality, only the following two cases: 

$ \bullet $ When $h \equiv 2 \pmod 4$ and $k \equiv 0 \pmod 4$, we use the values of $u_i, v_i, i=1, 2$, given  by \eqref{valuvtry2} and  we  take both $f, g$ to be arbitrary even integers and $t$ to be an arbitrary odd integer when  the values of both $x$ and $y$ given by \eqref{valxy} are integers. 

$\bullet$  Finally, when $h \equiv 2 \pmod 4$ and $k \equiv 2 \pmod 4$, we again use the values of $u_i, v_i, i=1, 2$, given  by \eqref{valuvtry2} and  we now take $f \equiv 2 \pmod 4, g \equiv 1 \pmod 4$ and $t \equiv 1 \pmod 4$, when also the values of both $x$ and $y$ given by \eqref{valxy} are integers. 

In each case we can easily express each of the three parameters $f, g$ and $t$ in terms of a new independent parameter  while ensuring that the relevant parity or congruence condition is  satisfied, and hence  when $h$ and $k$ are any arbitrary integers, whether positive or negative,  we can always obtain a three-parameter  solution in integers of the simultaneous Eqs.  \eqref{eqn}-\eqref{eqnk}, and thus  obtain infinitely many integers $n=x^2+y^2$ such that the integers $n, n+h$ and $n+k$ are all expressible as sums of two nonzero squares.  

\end{proof}

\section{Sequences of integers that are sums of two nonzero squares}
We will now construct several finite sequences of integers that may all be written as sums of two squares.

\subsection{Three consecutive integers expressible as sums of two squares}
As an immediate corollary to Theorem \ref{Thmhk}, we now obtain infinitely many integers $n$ such that $n-1, n, n+1$ are all sums of two nonzero squares. As this problem has been studied for over a century, we give below   our solution as a  theorem.
\begin{theorem}\label{Thconsectriple} If the integers $x, y$ and $n$ are defined in terms of three arbitrary integer parameters $p, q$ and $r$ by the relations,
\begin{equation}
\begin{aligned}
x & =  4r(16p^2r^2 + 8pr + 4r^2 + 1)q^2 - 8r(4p^2r^2 - 4p^2r + \\
& \quad \quad 2pr + r^2 - p - r)q - 2r(4p^2r - 2p^2 + 2p + r - 1), \\
y & = -2(4pr + 1)(16p^2r^2 + 8pr + 4r^2 + 1)q^2 \\
 & \quad \quad + 4(4pr + 1)(4p^2r^2 - 4p^2r+ 2pr + r^2 - p - r)q\\
  & \quad \quad + 16p^3r^2 - 8p^3r + 8p^2r + 4pr^2 - 2p^2 - 4pr - 1,\\
	n & = x^2+y^2,
\end{aligned}
\label{valxythreeconsecint}
\end{equation}
the  three consecutive integers $n-1, n$ and $ n+1$  are all expressible as sums of two squares. 
\end{theorem}
\begin{proof} The theorem follows immediately from Theorem \ref{Thmhk} by taking $(h, k, f, g, t)=(1, -1, 2p, 2q, 2r)$ and  using the relations \eqref{valxy}.
\end{proof}

By assigning numerical values to the parameters in the solution \eqref{valxythreeconsecint}, we will now derive  fairly simple solutions of our problem. As an example, taking $(q, r)=(0, 0)$ yields the solution $(x, y)=(0,  -2p^2 - 1)$ which gives us the three consecutive integers $(2p^2)^2 +(2p)^2,  (2p^2+1)^2+0^2, (2p^2+1)^2+1^2. $ This example has already been given in \cite{CD}. We now note that if we take $p=17m+5$, then 
\begin{align*}
(2p^2+1)^2 &=289(34m^2 + 20m + 3)^2 \\
 &=(15(34m^2 + 20m + 3))^2 +(8(34m^2 + 20m + 3))^2,
\end{align*} 
and all the three consecutive integers may be expressed as  sums of two nonzero squares.

As a second example, when $(p, q)=(0, 0)$, the relations \eqref{valxythreeconsecint} yield the solution $(x, y)=(-2r(r-1),  -1)$, leading to the three consecutive integers, 
\begin{equation}
(2r(r-1))^2 +0^2,\quad  (2r(r-1))^2 +1^2,  \quad (2r(r-1))^2 +2, \label{consectriple}
\end{equation}
where we note that 
\begin{equation}
(2r(r-1))^2 +2 =  (2r^2 - 2r - 1)^2 + (2r-1)^2. \label{identsectriple}
\end{equation}
 Here again, if we take $r=5m$, then the first integer of the sequence may be written as $100(m(5m-1))^2= (6m(5m-1))^2+(8m(5m-1))^2$ so that all the three consecutive integers are sums of two nonzero squares.  More generally,  if we take $r=(r_1^2 + r_2^2)m$ where $r_1, r_2$ and $m$ are arbitrary nonzero integers,  then also we can write the first term of the sequence as a  sum of two nonzero squares utilizing the fact that $(r_1^2+r_2^2)^2 = (r_1^2-r_2^2)^2 + (2r_1 r_2)^2$. 

\subsection{ The four integers $n, n+1, n+2, n+4$ expressible as sums of two squares}\label{n124}
The first term of the triple of consecutive integers \eqref{consectriple} is a perfect square, and hence taking $n=(2r(r-1))^2$, we immediately get four integers $n, n+1, n+2, n+4$ that are all expressible as sums of two squares. As noted earlier, if we take $r=5m$, all the four integers may be written as sums of two nonzero squares.

A more general value of $n$, such that $n, n+1, n+2, n+4$ are all  sums of two nonzero squares, is given by the following theorem.
 \begin{theorem}\label{Thmn124}
If we define the  integer $n$ in terms of arbitrary integer parameters $m$ and $r$ as
\begin{equation}
n= ((4m^2 - 4m + 2)r^2 - (8m^3 - 8m^2 + 2)r + 4m(m + 1)(m - 1)^2)^2,
\end{equation}
the four integers are $n, n+1, n+2, n+4$ are all  expressible as sums of two squares. Further, if $m$ and $r$ are both taken as multiples of $5$, all these four integers are sums of  squares of two nonzero integers.  
\end{theorem}
\begin{proof}
Any solution of the  diophantine equation,
\begin{equation}
x^2+2 =u^2+v^2, \label{eqxuv}
\end{equation}
in nonzero integers immediately  yields values of $x$ such that $x^2+1, x^2+2, x^2+4$ are all sums of two nonzero squares. If, in addition, $x$ is a multiple of $5$, then $x^2=25z^2$ where $z$ is an integer, and we get $x^2 =(3z)^2+(4z)^2$. Thus, taking   $n=x^2$, all the integers   $n, n+1, n+2, n+4$ are sums of two nonzero squares.

To solve Eq.\eqref{eqxuv}, we write,
\begin{equation}
\begin{aligned}
 x &=2m(m - 1)t+2r(r-1),\\
 u &=(2m^2 - 2m - 1)t+2r^2 - 2r - 1,\\
 v &=(2m-1)t+2r-1,
\end{aligned}
\label{subsxuv}
\end{equation}
when, in view of the identity \eqref{identsectriple}, Eq. \eqref{eqxuv} is readily solved and we get a nonzero solution for $t$, namely $t=2(m - r + 1)(m - r - 1)$ which, when substituted in \eqref{subsxuv}, yields the following solution of Eq. \eqref{eqxuv}
\begin{equation}
\begin{aligned}
x &= (4m^2 - 4m + 2)r^2 - (8m^3 - 8m^2 + 2)r \\
& \quad \quad + 4m(m + 1)(m - 1)^2,\\
u &= 4m(m - 1)r^2 - (8m^3 - 8m^2 - 4m + 2)r\\
& \quad \quad  + 4m^4 - 4m^3 - 6m^2 + 4m + 1, \\
v &= (4m - 2)r^2 - (8m^2 - 4m - 2)r \\
& \quad \quad + 4m^3 - 2m^2 - 4m + 1,
\end{aligned}
\label{solxuv}
\end{equation}
where $m$ and $r$ are arbitrary integer parameters. The value of $n$ stated in the theorem is precisely $x^2$ where $x$ is given by \eqref{solxuv}. Thus $n+1, n+2$ and $n+4$ are all expressible as sums of two nonzero squares. If both $m$ and $r$ are multiples of 5, then $x$ is also a multiple of $5$, and hence  $n$ is a multiple of $25$, and hence $n$  may also be written as a sum of  squares of 
two nonzero integers.
\end{proof}
In Theorem \ref{Thmn124} while we have stated, for simplicity, that $m$ and $r$ may be taken as multiples of 5, the theorem remains true when both $m$ and $r$ are multiples of any integer that is a sum of two squares.

\subsection{Infinitely many integers $n$ such that $n, n+1, n+2, n+4$ and $n+5$ are all sums of two nonzero squares}\label{n1245}
\begin{theorem}\label{Thmn1245} There exist infinitely many integers $n$ such that the  five integers $n, n+1, n+2, n+4, n+5$ are all sums of two nonzero squares.
\end{theorem} 
\begin{proof}
We will apply Theorem \ref{Thmhk} taking $h=2$ and $k=5$ to  obtain  infinitely many integers such that $n, n+2$ and $n+5$ are all sums of two squares. In fact, we will use the relations \eqref{valxy} and choose the parameters $u_i, v_i, i=1, 2$, such that the values of $x$ and $y$ given by \eqref{valxy} are integers, and moreover, $y=0$. We then get $n=x^2$, and hence $n+1=x^2+1^2$ and $n+4=x^2+2^2$ can be written as sums of two squares. We will also ensure that $x$ is a multiple of 5, that is, $x=5z$ where $z$ is an integer, so that $n=x^2=5^2z^2=(3z)^2+(4z)^2$. We will thus get infinitely many integers $n$ such that $n, n+1, n+2, n+4, n+5$ are all sums of two nonzero squares.  

On writing $(h, k, u_1,  v_1) =(2, 5, -1, -2)$, we get the following values of $x, y$ from \eqref{valxy}:
\begin{equation}
\begin{aligned}
x &= -(u_2^2v_2 - u_2v_2^2 + u_2 - v_2)/(2(2u_2 - v_2)), \\
y & = -(2u_2^2 - v_2^2 - 1)/(2(2u_2 - v_2)).
\end{aligned}
\end{equation}

We note that the numerator of $x$ may be written as follows:
\[(u_2^2 - 1)(2u_2 - v_2)-u_2(2u_2^2 - v_2^2 - 1).\]  
We now choose $u_2, v_2$ to be  integers such that 
\begin{equation}
2u_2^2 - v_2^2 -1=0, \label{equv}
\end{equation}
and  we then  get 
$(x, y)=((u_2^2-1)/2, 0)$. It follows from congruence considerations modulo 2  that  both $u_2, v_2$ must be odd integers, and hence  $x$ will be an integer.  As an example, the solution $(u_2, v_2)=(29, 41)$ of Eq. \eqref{equv} yields $(x, y) =(420, 0)$. This yields an example with $n=420^2$ such that the integers $n, n+1, n+2, n+4$ and $n+5$ are all expressible as sums of two squares of nonzero integers. 

Eq. \eqref{equv} may be re-written, on writing $(u_2, v_2)= (\alpha, \beta)$, as  the classical negative Pell equation 
\begin{equation}
\beta^2-2 \alpha^2=-1, \label{negPell}
\end{equation}
and here we need solutions of \eqref{negPell}  that satisfy the additional condition $\alpha^2 \equiv 1 \pmod 5$ so that  the resulting values of $x$ are multiples of 5.
 
The theory of the negative Pell equation is well-known (see \cite[pp.\ 201--204]{Na}). Eq. \eqref{negPell} has infinitely many solutions, its  fundamental solution  is  $(\alpha_1, \beta_1)= (1, 1)$, and all  solutions $(\alpha_r, \beta_r)$ in positive integers may be obtained from the relations    $\beta_r+ \alpha_r \sqrt{2} = (1+ \sqrt{2})^{2r-1}, r=1, 2, 3, \ldots $. 

We note that $\beta_2+ \alpha_2 \sqrt{2}= 7+5 \sqrt{2}$ and $\beta_3+ \alpha_3 \sqrt{2}= 41+29 \sqrt{2}$.  While the solutions $(\alpha_1, \beta_1)= (1, 1)$ and $(\alpha_3, \beta_3)= (29, 41)$ of Eq. \eqref{negPell} satisfy the congruence condition $\alpha^2 \equiv 1 \pmod 5$, the second solution $(\alpha_2, \beta_2)= (5, 7)$ does not satisfy this condition.

If, for any positive integer value of $m$, the solution  $(\alpha_m, \beta_m)$  of Eq. \eqref{negPell} satisfies the congruence condition $\alpha^2 \equiv 1 \pmod 5$, we shall show that the solution $(\alpha_{3+m}, \beta_{3+m})$  also satisfies the condition  $\alpha^2 \equiv 1 \pmod 5$. Since
\begin{multline*}
\beta_{3+m}+ \alpha_{3+m} \sqrt{2} =(1+ \sqrt{2})^{6+2m-1} =(1+ \sqrt{2})^6(1+ \sqrt{2})^{2m-1}\\
\quad \quad = (99 + 70 \sqrt{2})(\beta_m+ \alpha_m \sqrt{2}) = (140 \alpha_m+99\beta_m) + (99 \alpha_m +70 \beta_m) \sqrt{2},
\end{multline*}
we get $\alpha_{3+m}=99 \alpha_m +70 \beta_m$.  Hence $\alpha^2_{3+m} \equiv (99 \alpha_m)^2 \equiv \alpha^2_m \equiv 1 \pmod 5$. Thus,  the solution $(\alpha_{3+m}, \beta_{3+m})$   satisfies the congruence condition  $\alpha^2 \equiv 1 \pmod 5$,
and it now follows by induction that for all positive integer values of $k$, the solution $(\alpha_{3k+m}, \beta_{3k+m})$  satisfies the congruence condition  $\alpha^2 \equiv 1 \pmod 5$. 

We have already seen that when $m=1$ and  also when $m=3$, the solution  $(\alpha_m, \beta_m)$  satisfies the condition  $\alpha^2 \equiv 1 \pmod 5$. It follows that for all positive integer values of $k$, the solutions $(\alpha_{3k+1}, \beta_{3k+1})$ and $(\alpha_{3k+3}, \beta_{3k+3}) $  satisfy the congruence condition  $\alpha^2 \equiv 1 \pmod 5$. We can thus generate infinitely many integer solutions of Eq. \eqref{negPell}  such that $\alpha^2 \equiv 1 \pmod 5$. These solutions yield infinitely many examples of finite sequences of integers $n, n+1, n+2, n+4$ and $n+5$ that are all expressible as sums of two squares of nonzero integers. 
\end{proof}

We have already noted  two  integer solutions of Eq. \eqref{negPell} that  satisfy the congruence condition $\alpha^2 \equiv 1 \pmod 5$ namely, $(\alpha_1, \beta_1)=(1, 1) $ and $(\alpha_3, \beta_3)= (29, 41)$. While the solution $(\alpha_1, \beta_1)$ yields $x=0$ and we get $n=0$, the solution $(\alpha_3, \beta_3)$ yields $x=420$ and hence we get $n=420^2=176400$  such that the five integers $n, n+1, n+2, n+4$ and $n+5$ are expressible as  sums of two nonzero squares. In fact, we have,
\begin{align*}
176400&=252^2+336^2, &176401& =1^2+420^2, \\
176402&=29^2+419^2, & 176404&=2^2+420^2,  \\
 176405 &=41^2+418^2.  && 
\end{align*}

The next four smallest integer solutions of Eq. \eqref{negPell} that  satisfy the congruence condition $\alpha^2 \equiv 1 \pmod 5$ are as follows:
\[
(169, 239), \quad (5741, 8119), \quad (33461, 47321), \quad (1136689, 1607521).
\]
These solutions yield the following four values of $n$ such that the integers $n, n+1, n+2, n+4$ and $n+5$ are all expressible as  sums of two nonzero squares:
\[
14280^2, \quad 16479540^2, \quad 559819260^2, \quad 646030941360^2.
\]

We give below the representations of the five integers $n, n+1, n+2, n+4$ and $n+5$ as  sums of two nonzero squares when $n=14280^2=203918400$ and $n=16479540^2=271575238611600$:
\begin{align*}
203918400 & = 8568^2+11424^2, \\
203918401 & = 1^2+14280^2,\\
203918402 & = 169^2+14279^2,\\
203918404 & = 2^2+14280^2,\\
203918405 & = 239^2+14278^2,
\end{align*}
and
\begin{align*}
271575238611600 & = 9887724^2+ 13183632^2, \\
271575238611601 & = 1^2+16479540^2,\\
271575238611602 &= 5741^2+16479539^2, \\
271575238611604 & = 2^2+16479540^2,\\
271575238611605 &=8119^2+16479538^2.
\end{align*}

\subsection{Infinitely many integers $n$ such that all five terms of the arithmetic progression $n, n+4, \ldots, n+16$ are   sums of two nonzero squares}
\begin{theorem}\label{ThmAP}
There exist infinitely many integers $n$ such that all five terms of the arithmetic progression $n, n+4, \ldots, n+16$ are   sums of two nonzero squares. 
\end{theorem}
\begin{proof}
We will follow  a procedure similar to that of Section \ref{n1245} to  obtain integers $n$ with the desired property.

We will take $(h, k)=(8, 12)$ and apply Theorem \ref{Thmhk} by choosing the parameters $u_i, v_i, i=1, 2$, such that the relations \eqref{valxy} yield an integer value of $x$, and $y=0$. We will thus get $n=x^2$ and hence $n+4, n+8, n+12$ and $n+16$ can all be  written as sums of two nonzero squares. We will also ensure that $x$ is a multiple of 37, that is, $x=37z$, and hence $n=x^2=37^2z^2=(12z)^2+(35z)^2$. Thus all the five integers $n, n+4, \ldots, n+16$ can be written as sums of two nonzero squares. We will show that the method generates infinitely many integers $n$ with this property.  The choice of the integer 37 is based on numerical data obtained by computer trials. 

On writing $(h,k,u_1, v_1)=(8, 12, -1, -3)$ the relations \eqref{valxy} yield the following values of $x$ and $y$:
\begin{equation}
\begin{aligned}
x &= -(u_2^2v_2 - u_2v_2^2 + 3u_2 - 7v_2)/(2(3u_2 - v_2)),\\
y & = -(3u_2^2 - v_2^2 - 18)/(2(3u_2 - v_2)).  
\end{aligned}
\label{tryxyAP}
\end{equation}

The numerator of $x$ may be written as $(3u_2^2 - v_2^2 - 18)u_2 - (u_2^2 - 7)(3u_2 - v_2)$, hence if we choose the parameters $u_2, v_2$ such that 
\begin{equation}
3u_2^2 - v_2^2 - 18=0. \label{PellAP}
\end{equation}
the values of $x, y$ given by \eqref{tryxyAP} may be written as follows: 
\begin{equation}
x=  (u_2^2 - 7)/2,\quad  y=0. \label{tryxyAPsec}
\end{equation}
It follows from congruence considerations modulo 4 that if $(u_2, v_2)$ is a solution in integers of Eq. \eqref{PellAP}, then both $u_2$ and $v_2$ must be odd integers, and hence the value of $x$ given by \eqref{tryxyAPsec} is necessarily an integer. 

Now Eq. \eqref{PellAP} may be re-written, on writing $(u_2, v_2)= (\alpha, \beta)$, as  the generalized Pell equation 
\begin{equation}
\beta^2-3 \alpha^2=-18, \label{Pellgen}
\end{equation}
and here we need solutions of \eqref{Pellgen}  that satisfy the additional condition $\alpha^2 \equiv 7 \pmod {37}$ so that  the resulting values of $x$ are multiples of 37. 
 
The smallest solution in positive integers of Eq. \eqref{Pellgen} is $(\alpha, \beta)=(3, 3)$, and the fundamental solution of the related Pell equation $\beta^2-3 \alpha^2=1$ is $(\alpha, \beta)=(1, 2)$. It, therefore, follows from the theory of the generalized Pell equation (see \cite[pp.\ 204--209]{Na}) that infinitely many  solutions $(\alpha_r, \beta_r)$ of Eq. \eqref{Pellgen} may be obtained from the relations $\beta_r+\alpha_r \sqrt{3}=(3+3 \sqrt{3})(2+ \sqrt{3})^r, r=1, 2, 3, \ldots $. We note that  $\beta_1+\alpha_1 \sqrt{3}=15 + 9 \sqrt{3}$, and the solution $(\alpha_1, \beta_1) = (9, 15)$ satisfies the congruence condition $\alpha^2 \equiv 7 \pmod {37}$.

If, for any positive integer value of $m$, the solution  $(\alpha_m, \beta_m)$  of Eq. \eqref{Pellgen} satisfies the congruence condition $\alpha^2 \equiv 7 \pmod {37}$, we shall show that the solution $(\alpha_{18+m}, \beta_{18+m})$ also satisfies the condition  $\alpha^2 \equiv 7 \pmod {37}$. Since
\begin{equation*}
\begin{aligned}
\beta_{18+m}+\alpha_{18+m} \sqrt{3}& =(3+3 \sqrt{3})(2+ \sqrt{3})^{18+m} \\
 & =(2+ \sqrt{3})^{18} (3+3 \sqrt{3})(2+ \sqrt{3})^m\\
 & =(9863382151 + 5694626340 \sqrt{3})(\beta_m+\alpha_m \sqrt{3}) \\
 & =17083879020 \alpha_m + 9863382151 \beta_m \\
& \quad \quad +(9863382151 \alpha_m+ 5694626340 \beta_m) \sqrt{3},
\end{aligned}
\end{equation*}
we get $\alpha_{18+m}=9863382151 \alpha_m+ 5694626340 \beta_m$. We note that $9863382151 \equiv -1 \pmod {37}$ and   $5694626340 \equiv 0 \pmod{37}$, hence $\alpha^2_{18+m} \equiv \alpha^2_m \equiv 7 \pmod {37}$. Thus,  the solution $(\alpha_{18+m}, \beta_{18+m})$   satisfies the congruence condition  $\alpha^2 \equiv 7 \pmod {37}$,
and it now follows by induction that for all positive integer values of $k$,  the  solution $(\alpha_{18k+m}, \beta_{18k+m})$ of Eq. \eqref{Pellgen}  satisfies the  congruence condition $\alpha^2 \equiv 7 \pmod {37}$. 

We have already seen that the solution  $(\alpha_1, \beta_1)$  satisfies the condition  $\alpha^2 \equiv 7 \pmod {37}$. It follows that for all positive integer values of $k$, the solutions $(\alpha_{18k+1}, \beta_{18k+1})$ satisfy  the congruence condition  $\alpha^2 \equiv 7 \pmod {37}$. We can thus generate infinitely many integer solutions of the generalized Pell Eq. \eqref{Pellgen}  such that $\alpha^2 \equiv 7 \pmod {37}$, and these solutions yield infinitely many values of $x$ that are multiples of 37. These solutions thus yield infinitely many examples of finite sequences of integers $n, n+4, n+8, n+12$ and $n+16$ that are all expressible as sums of two squares of nonzero integers. 
\end{proof}

As numerical examples,   the first three values of $n$ that we get such that all the five integers $n+4i, i=1, \ldots, 4$, are expressible as sums of two nonzero squares are $37^2$,  $15171049214426911911337^2$ and
\begin{equation}
5903741433259753755776680512005460787523437^2. \label{largen}
\end{equation}
As the representations of $n+4$ and $n+16$ as sums of two squares are obvious, we just give below the representations of $n, n+8$ and $n+12$ as sums of two nonzero squares.

When $n=37^2$, we have 
\[
37^2=12^2+35^2, 37^2+8= 36^2+9^2, 37^2+12=15^2+34^2,
\]
when $n=15171049214426911911337^2$, we have
\[
\begin{aligned}
n & = 4920340285760079538812^2
 \\ & \quad \quad +14350992500133565321535^2, \\
n+8 & =174189834459^2 +15171049214426911911336^2,\\
n+12 & =301705643445^2 +15171049214426911911334^2.
\end{aligned}
\]
and when $n$ is given by \eqref{largen}, we have
\[
\begin{aligned}
&n=1914726951327487704576220706596365660818412^2\\
&+ 5584620274705172471680643727572733177387035^2,&\\
&n+8 =3436201808177090682609^2\\
& +5903741433259753755776680512005460787523436^2,\\
&n+12 = 5951676116822766300375^2\\
& +5903741433259753755776680512005460787523434^2.
\end{aligned}
\]

\section{Concluding remarks}

In this paper,   given any two arbitrary integers $h$ and $k$,  we established a three-parameter solution to the problem of finding integers $n$ such that $n,  n+h$ and $n+k$ are all representable as sums of squares of two nonzero integers.  Further, we found parametric solutions to the problems of finding integers $n$ such that the three consecutive integers $n-1, n, n+1$, and all the four integers $n, n+1, n+2, n+4$, are representable as sums of squares of two nonzero integers. We also found infinitely many integers $n$ such that the integers $n,  n+1,  n+2,  n+4$ and $n+5$ are all expressible as sums of two nonzero squares, and similarly, we found infinitely many integers $n$ such that all the terms of the arithmetic progression  $n+4j,  j=0, 1,  \cdots, 4, $ are sums of two nonzero squares.

{\bf Acknowledgement:} We sincerely thank the anonymous referee for carefully reading our manuscript and giving constructive suggestions that have led to improvements in the paper. The second author wishes to thank SERB for the MATRICS grant MTR/2022/000545.

\end{document}